\documentclass[11pt,a4paper]{article}

\usepackage[english]{babel}
\usepackage{times}
\usepackage{graphicx}
\usepackage{amscd}
\usepackage{amsmath}
\usepackage{amsrefs}
\usepackage{amsfonts}
\usepackage{amssymb}
\usepackage{amsthm}
\usepackage{latexsym}

\newtheorem{theorem}{Theorem}[section]

\newtheorem{lemma}[theorem]{Lemma}

\newtheorem{remark}[theorem]{Remark}

\numberwithin{equation}{section}

\newcommand{\R}{\mathbb{R}}
\newcommand{\N}{\mathbb{N}}
\renewcommand{\epsilon}{\varepsilon}
\newcommand{\eps}{\varepsilon}
\newcommand{\e}{\varepsilon}

\newcommand{\ugu}{\;{\stackrel{k}{=}}\;}

\renewcommand{\le}{\leqslant}

\renewcommand{\ge}{\geqslant}

\usepackage{authblk}

\begin{document}

\author[1]{Serena Dipierro}

\author[2]{Ovidiu Savin}

\author[1]{Enrico Valdinoci}

\affil[1]{\footnotesize Department of Mathematics and Statistics,
University of Western Australia,
35 Stirling Highway, WA6009 Crawley,
Australia \medskip } 

\affil[2]{\footnotesize Department of Mathematics, Columbia University,
2990 Broadway,
New York NY 10027, USA \medskip}

\title{On divergent fractional Laplace equations\thanks{The first and third authors are member of
INdAM and are supported by the Australian
Research Council Discovery Project DP170104880 NEW ``Nonlocal Equations at Work''.
The first author
is supported by the
Australian Research Council DECRA DE180100957 ``PDEs, free boundaries and
applications''. The second author is supported by the National Science Foundation grant
DMS-1500438.
Emails:
{\tt serena.dipierro@uwa.edu.au}, {\tt savin@math.columbia.edu}, {\tt enrico.valdinoci@uwa.edu.au} }}

\date{}

\maketitle

\begin{abstract}
We consider the divergent fractional Laplace operator presented in~\cite{POLINOMI}
and we prove three types of results.

Firstly, we show that any given function can be locally shadowed by a solution
of a divergent fractional Laplace equation which is also prescribed in a neighborhood of infinity.

Secondly, we take into account the Dirichlet problem for the divergent fractional Laplace equation,
proving the existence of a solution and characterizing its multiplicity.

Finally, we take into account the case of nonlinear equations, obtaining
a new approximation results.

These results maintain their interest also in the case
of functions for which the fractional Laplacian can be defined
in the usual sense.
\end{abstract}

\bigskip\bigskip

\section{Introduction}

Given~$u:\R^n\to\R$ and~$s\in(0,1)$,
to define the fractional Laplacian of~$u$,
\begin{equation}\label{CLfar} (-\Delta)^s u(x):=\lim_{\rho\searrow0}\int_{\R^n\setminus B_\rho(x)}
\frac{u(x)-u(y)}{|x-y|^{n+2s}}\,dy,\end{equation}
one typically needs two main requisites on the function~$u$:
\begin{itemize}
\item $u$ has to be sufficiently smooth in the vicinity of~$x$,
for instance~$u\in C^\gamma(B_\delta(x))$ for some~$\delta>0$ and~$\gamma>2s$,
\item $u$ needs to have a controlled growth at infinity, for instance
\begin{equation}\label{GR01d}
\int_{\R^n}\frac{|u(x)|}{1+|x|^{n+2s}}\,dx<+\infty.
\end{equation}
\end{itemize}
Nevertheless, in~\cite{POLINOMI} we have recently introduced
a new notion of ``divergent'' fractional Laplacian,
which can be used even when condition~\eqref{GR01d} is violated.
This notion takes into account the case of functions with
polynomial growth, for which the classical definition in~\eqref{CLfar}
makes no sense, and it recovers the classical definition for functions
with controlled growth such as in~\eqref{GR01d}.\medskip

The notion of divergent fractional Laplacian possesses
several interesting features and technical advantages,
including suitable Schauder estimates
in which the full smooth
H\"older norm of the solution is controlled by a suitable
seminorm of the nonlinearity. Moreover, compared
to~\eqref{CLfar},
the notion of divergent fractional Laplacian
is conceptually closer to the classical case
in the sense that it requires a sufficient degree
of regularity of the function~$u$ at a given point,
without global conditions (up to a mild control at
infinity of polynomial type), thus attempting to make
the necessary requests as close as possible to  the case
of the classical Laplacian.\medskip

In this article, we consider the setting of the divergent Laplacian
and we obtain the following results:
\begin{itemize}
\item an approximation result with solutions
of divergent Laplacian equations: we will show that
these solutions can locally shadow any prescribed function,
maintaining also a complete prescription at infinity,
\item a characterization of the Dirichlet problem: we will show
that the (possibly inhomogeneous) Dirichlet problem is solvable
and we determine the multiplicity of the solutions,
\item an approximation result with solutions of nonlinear divergent Laplacian equations,
up to a small error also in the forcing term. 
\end{itemize}

To state these results in detail,
we now recall the precise framework for the divergent fractional
Laplacian.
Given~$k\in\N$, we consider the space of functions 
$$ {\mathcal{U}}_k:=
\left\{ u:\R^n\to\R, {\mbox{ s.t. $u$ is continuous in $B_1$ and }}
\int_{\R^n}\frac{|u(x)|}{1+|x|^{n+2s+k}}\,dx<+\infty
\right\}.$$
Then (see Definition~1.1 in~\cite{POLINOMI}) we use the notation
$$ \chi_R(x):=\begin{cases}
1 & {\mbox{ if }}x\in B_R,\\
0 & {\mbox{ if }}x\in \R^n\setminus B_R,
\end{cases}$$
and
we say that
\begin{equation}\label{RGAs} (-\Delta)^su\ugu f\qquad{\mbox{ in }}B_1\end{equation}
if there exist a family of polynomials~$P_R$, which have degree at most~$k-1$,
and functions~$f_R : B_1\to\R$ such
that~$(-\Delta)^su=f_R+P_R$
in~$B_1$ in the viscosity sense, with
$$ \lim_{R\to+\infty}f_R(x) = f(x)$$
for any~$ x\in B_1$.\medskip

Interestingly, one can also think that
the right hand side of
equation~\eqref{RGAs} is not just a function, but an equivalence class of
functions modulo polynomials, since one can freely
add to~$f$ a polynomial of degree~$k$ when~$s\in\left(0,\frac12\right]$
and of degree~$k+1$ when~$s\in\left(\frac12,1\right)$
(see Theorem~1.5 in~\cite{POLINOMI}).
\medskip

The first result that we provide in this setting
states that every given function
can be modified in an arbitrarily small way in~$B_1$,
remaining unchanged in a large ball, in such a way
to become $s$-harmonic with respect to the divergent fractional Laplacian.

\begin{theorem}\label{ALL}[All divergent functions are locally
$s$-harmonic up to a small error]
Let~$k$, $m\in\N$ and~$u:\R^n\to\R$ be such that~$u\in
C^m(\overline{B_1})$ and
\begin{equation*} 
\int_{B_1^c}\frac{|u(x)|}{|x|^{n+2s+k}}\,dx<+\infty.\end{equation*}
Then, for any~$\eps>0$ there exist~$u_\eps$ and~$R_\eps>1$
such that
\begin{eqnarray}\label{DES:ALL1}
&& (-\Delta)^s u_\eps \ugu 0 \quad{\mbox{ in }}B_1,\\
\label{DES:ALL2}
&& \| u_\eps-u\|_{C^m({B_1})}\le \eps\\
\label{dtsfgyvhoqwfyguywqegfiowel}
{\mbox{and }}&& u_\eps=u \quad{\mbox{ in }}B_{R_\eps}^c.
\end{eqnarray}
\end{theorem}

A graphical sketch of Theorem~\ref{ALL} is given in Figure~\ref{fig:1}
(notice the possible wild oscillations of~$u_\e$ in~$B_{R_\e}\setminus B_1$).

\begin{figure}[htbp]
  \centering
  \includegraphics[width=0.6\linewidth]{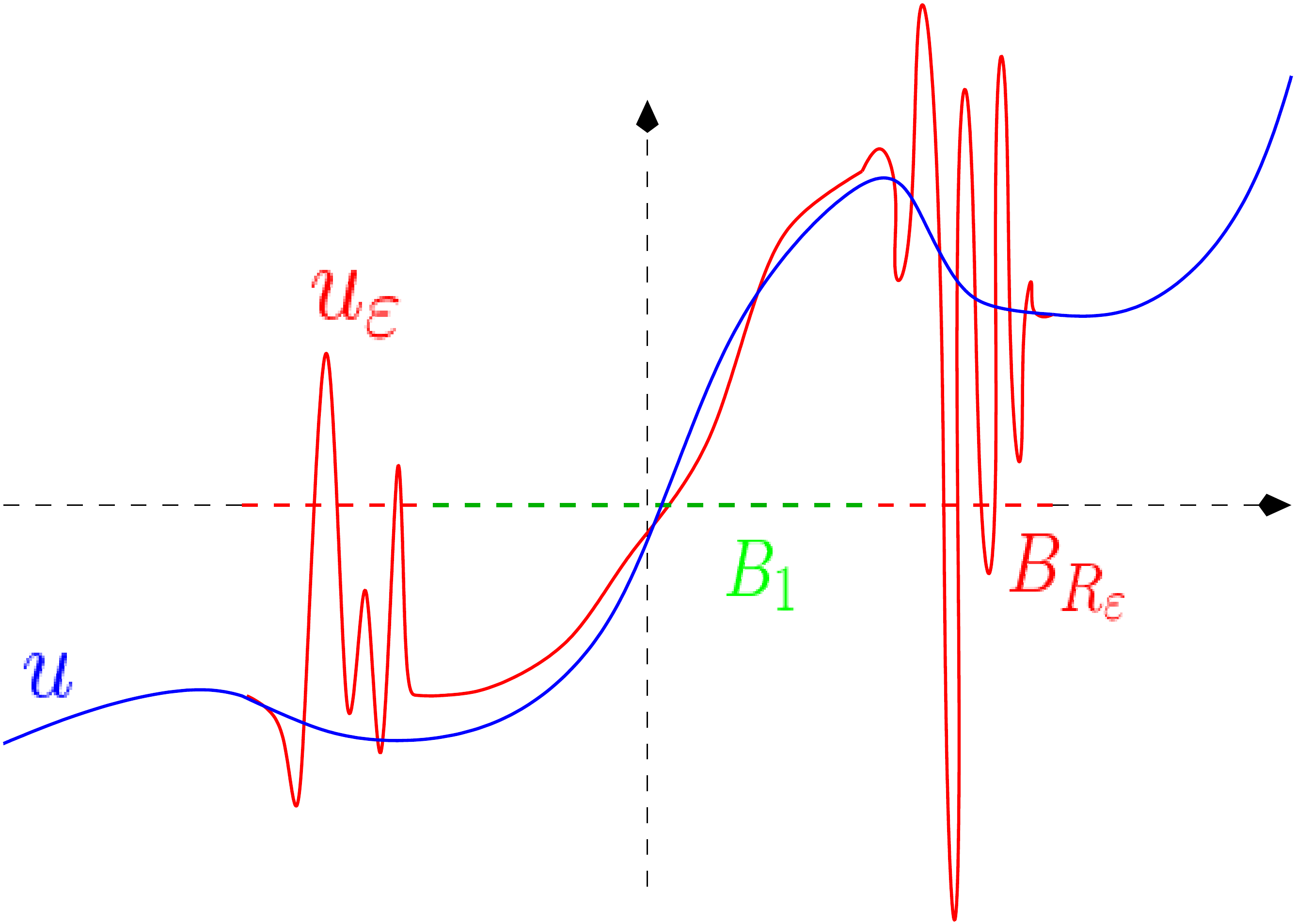}
  \caption{\it {{The approximation result in Theorem~\ref{ALL}.}}}
    \label{fig:1}
\end{figure}

\begin{remark}{\rm
When~$k=0$ and~$u=0$ outside~$B_2$,
Theorem~\ref{ALL} reduces to the main result of~\cite{ALL-FUNCTIONS}.
Interestingly, in the case considered here, one
can preserve the values of the given function~$u$ at infinity
and, if the growth of~$u$ at infinity is ``too fast'' for the classical
fractional Laplacian to be defined, then the result still carries on,
in the divergent fractional Laplace setting.}
\end{remark}

\begin{remark}{\rm We observe that Theorem~\ref{ALL}
does not hold under the additional assumption that
\begin{equation}\label{STdffOA0}
{\mbox{$|u_\e(x)|\le P(x)$ for all~$x\in\R^n$,}}\end{equation}
for a given polynomial~$P$ (that is, one cannot
replace a growth assumption at infinity
with a pointwise bound).
Indeed, under assumption~\eqref{STdffOA0},
we have that
$$ \int_{\R^n} \frac{|u_\e(x)|}{1+|x|^{n+2s+d}}\,dx\le
\int_{\R^n} \frac{|P(x)|}{1+|x|^{n+2s+d}}\,dx =:J<+\infty,$$
being~$d\in\N$ the degree of the polynomial~$P$.
As a consequence of this and~\eqref{DES:ALL1},
we deduce from Theorem~1.3 of~\cite{POLINOMI}
that for any~$\gamma>0$ such that~$\gamma$ and~$\gamma+2s$ are not integer,
$$ \|u_\e\|_{C^{\gamma+2s}}(B_{1/2})\le C\,J,$$
for some~$C$ depending only on~$J$,~$n$, $s$, $\gamma$
and~$d$.
In particular, if~$\gamma+2s\ge m$,
we would have from~\eqref{DES:ALL2} that
$$ \e\ge \|u_\e-u\|_{C^m(B_1)}\ge \|u_\e-u\|_{C^m(B_{1/2})}
\ge \|u\|_{C^m(B_{1/2})}-\|u_\e\|_{C^m(B_{1/2})}\ge
\|u\|_{C^m(B_{1/2})}-C\,J.$$
This set of inequalities would be violated for~$\e\in(0,1)$ by any function~$u$
satisfying
\begin{equation}\label{STdffOA} \|u\|_{C^m(B_{1/2})}\ge C\,J+1.\end{equation}
That is, solutions with a large $C^m$-norm (more specifically
with a norm as in~\eqref{STdffOA}) cannot be approximated arbitrarily well by $s$-harmonic functions (not even ``modulo polynomials'') that satisfy a polynomial bound
as in~\eqref{STdffOA0}.

Interestingly, this remark
is independent from~$R_\e$ in~\eqref{dtsfgyvhoqwfyguywqegfiowel}
(hence, it is not possible to arbitrarily improve
the approximation results if we require an additional polynomial bound, even if we drop the request that the approximating
function is compactly supported).
}\end{remark}

\medskip

Theorem~\ref{ALL} is also related to some recent results
in~\cite{MR3716924, MR3774704, MR3935264, KRYL, CAR, CARBOO}
(see~\cite{MR3790948} for an elementary exposition in the case
of the fractional Laplacian in dimension~1).
\medskip

Next result focuses on the
Dirichlet problem for divergent fractional Laplacians.
We show that, given an external datum and a forcing term,
the Dirichlet problem has a solution.
Differently from the classical case,
when~$k\not=0$ such solution is not unique, and we determine
the dimension of the multiplicity space.

\begin{theorem}\label{DIRI}[Solvability of the Dirichlet problem for divergent fractional
Laplacians]
Let~$k\in\N$ and~$u_0:B_1^c\to\R$ be such that
\begin{equation*} 
\int_{B_1^c}\frac{|u_0(x)|}{|x|^{n+2s+k}}\,dx<+\infty.\end{equation*}
Let~$f$ be continuous in~$B_1$.
Then, there exists a
function~$u\in{\mathcal{U}}_k$ such that
\begin{equation}\label{DIRI2}
\left\{
\begin{matrix}
(-\Delta)^s u \ugu f \quad{\mbox{ in }}\;B_1,\\
u = u_0
\quad{\mbox{ in }}\;B_1^c.
\end{matrix}
\right.
\end{equation}
Also, the space of solutions of~\eqref{DIRI2} has dimension~$N_k$,
with
\begin{equation}\label{NK} N_k:=
\sum_{j=0}^{k-1}\binom{j+n-1}{n-1}.\end{equation}
\end{theorem}

With the aid of Theorems~\ref{ALL}
and~\ref{DIRI}, we can also consider the case of nonlinear equations,
namely the case in which the right hand side depends also
on the solution (as well as on its derivatives, since
the result that we provide is general enough to comprise such a case too).

In this setting, we establish that any prescribed function
satisfies any prescribed nonlinear
(and possibly divergent) fractional Laplace equation, up to an arbitrarily small error,
once we are allowed to make arbitrarily small modifications
of the given function in a given region, preserving its values
at infinity. The precise result that we have is the following one:

\begin{theorem}\label{NONLI}[All divergent functions almost
solve nonlinear equations]
Let~$k$, $m\in\N$ and~$u:\R^n\to\R$ be such that~$u\in
C^{2m}(\overline{B_1})$ and
\begin{equation*} 
\int_{B_1^c}\frac{|u(x)|}{|x|^{n+2s+k}}\,dx<+\infty.\end{equation*}
Let
$$ N(m):= n+\sum_{j=0}^m n^j$$
and let~$F\in C^m(\R^{N(m)})$.

Then, for any~$\eps>0$ there exist~$u_\eps$,
$\eta_\e:\R^n\to\R$ and~$R_\eps>1$
such that
\begin{eqnarray}
\label{AMLO-1}&& (-\Delta)^s u_\eps(x) \ugu F\big(x,u_\e(x),\nabla u_\e(x),\dots,D^m u_\e(x)\big)
+\eta_\e(x) \quad{\mbox{ for all }}x\in B_1,\\
\label{AMLO-2}
&& \| \eta_\e\|_{L^\infty(B_1)}\le \eps,\\
\label{AMLO-3}
&& \| u_\eps-u\|_{C^m({B_1})}\le \eps\\
\label{AMLO-4}
{\mbox{and }}&& u_\eps=u \quad{\mbox{ in }}B_{R_\eps}^c.
\end{eqnarray}
\end{theorem}

\begin{remark} {\rm We think that it is a very interesting {\em open problem}
to determine whether the statement in Theorem~\ref{NONLI}
holds true also with~$\eta_\e:=0$.
This would give that any given function can be locally approximated
arbitrarily well by functions which solve exactly (and not only approximatively)
a nonlinear equation.}\end{remark}

\begin{remark} {\rm All the results presented here
maintain their own interest even in the case~$k=0$:
in this case, the definition of divergent fractional Laplacian
boils down to the usual fractional Laplacian (see Corollary~3.8
in~\cite{POLINOMI}).}\end{remark}
\medskip

The rest of this article is organized as follows.
In Section~\ref{SF-MAJo1} we give the proof of Theorem~\ref{ALL},
in Section~\ref{SF-MAJo2} we deal with the proof of Theorem~\ref{DIRI},
and in Section~\ref{SF-MAJo3} we focus on Theorem~\ref{NONLI}.

\section{Proof of Theorem~\ref{ALL}}\label{SF-MAJo1}

To prove Theorem~\ref{ALL}, we first present an observation
on the decay of the divergent fractional Laplacians for functions that
vanish on a large ball:

\begin{lemma}\label{LT8}
Let~$k\in\N$ and~$R>3$.
Let~$u:\R^n\to\R$ be such that~$u=0$ in~$B_R$ and
\begin{equation}\label{Dh0} 
\int_{\R^n}\frac{|u(x)|}{1+|x|^{n+2s+k}}\,dx<+\infty.\end{equation}
Then, there exists~$f:B_1\to \R$ such that~$(-\Delta)^s u\ugu f$
in~$B_1$ and for which the following statement holds true:
for any~$\eps>0$
and any~$m\in\N$, there exists~${\bar{R}_\eps}>3$ such that
if~$R\ge {\bar{R}_\eps}$ then
\begin{equation}\label{Dh}
\| f\|_{C^m(B_{1})}\le \eps.\end{equation}
\end{lemma}

\begin{proof} {F}rom Remark~3.5 in~\cite{POLINOMI}, we can write
that~$(-\Delta)^s u\ugu f$ in~$B_1$, with
\begin{eqnarray*}&& f(x)=f_u(x):=
\int_{B_2} \frac{u(x)-u(y)}{|x-y|^{n+2s}}\,dy+
\int_{B_2^c} \frac{u(x)}{|x-y|^{n+2s}}\,dy+
\int_{B_2^c} \frac{u(y)\;\psi(x,y)}{|y|^{n+2s+k}}\,dy\\ &&\qquad\qquad\qquad=
\int_{B_R^c} \frac{u(y)\;\psi(x,y)}{|y|^{n+2s+k}}\,dy,\end{eqnarray*}
for some function~$\psi$ satisfying, for any~$j\in \N$,
$$\sup_{{x\in B_1},\,{y\in B_2^c}} 
|D^j_x \psi(x,y)| \le C_j,$$
for some~$C_j>0$.
In particular, for any~$x\in B_1$,
$$ |D^jf(x)|\le \int_{B_R^c} \frac{|u(y)|\;|D^j_x\psi(x,y)|
}{|y|^{n+2s+k}}\,dy\le C_j\,\int_{B_R^c} \frac{|u(y)|
}{|y|^{n+2s+k}}\,dy,$$
so the desired claim in~\eqref{Dh} follows from~\eqref{Dh0}.
\end{proof}

With this, we complete the proof of Theorem~\ref{ALL} in the following way.

\begin{proof}[Proof of Theorem~\ref{ALL}]
{F}rom Theorem~1.1 of~\cite{ALL-FUNCTIONS} we know that
there exist a function~$v_\eps$ and~$\rho_\eps>1$
such that
\begin{eqnarray}
\label{veps1}
&& (-\Delta)^s v_\eps = 0 \quad{\mbox{ in }}B_1,\\
\label{90hx:2}
&& \| v_\eps-u\|_{C^m({B_1})}\le \eps\\
\label{90hx:3}
{\mbox{and }}&& v_\eps=0 \quad{\mbox{ in }}B_{\rho_\eps}^c.
\end{eqnarray}
For any~$R>3$, we also set~$\tilde u_R:= (1-\chi_R)\,u$.
Notice that
\begin{equation}\label{90hx:9876}
\tilde u_R=u\quad{\mbox{ in }}B_R^c.\end{equation}
In addition,
\begin{equation}\label{90hx}
\tilde u_R=0\quad{\mbox{ in }}B_R,\end{equation} so, in view of
Lemma~\ref{LT8}, there exist
a function~$f_{\eps}$ and~${\bar{R}_\eps}>3$ such that
\begin{eqnarray}
\label{u6s8}&& (-\Delta)^s \tilde u_{{\bar{R}_\eps}} \ugu f_\eps
\quad{\mbox{ in }}B_2,\\ \label{fduyas}
{\mbox{and }}&& \|  f_\eps \|_{C^m(B_{2})}\le \eps.
\end{eqnarray}
Now we consider the standard solution of the Dirichlet
problem
\begin{equation}\label{weps1} \left\{
\begin{matrix}
(-\Delta)^s w_\eps=f_\eps \quad{\mbox{ in }}B_2,\\
w_\eps=0 \quad{\mbox{ in }}B_2^c.
\end{matrix}
\right.\end{equation}
{F}rom 
Proposition~1.1 in~\cite{ROS-JMPA}, we have that
\begin{eqnarray} \label{67dff}
\|w_\eps\|_{C^s(\R^n)}\le C\, \|f_\eps\|_{L^\infty(B_2)},\end{eqnarray}
for some~$C>0$. 

Now we take~$\gamma:=m-s$. Notice that~$\gamma\not\in\N$
and~$\gamma+2s=m+s\not\in\N$.
Then, by Schauder estimates (see e.g.
Theorem~1.3 in~\cite{POLINOMI}, applied here with~$k:=0$),
and exploiting~\eqref{fduyas}
and~\eqref{67dff}, possibly renaming~$C>0$ line after line,
we obtain that
\begin{equation}\label{90hx:1}
\begin{split}
&
\|w_\eps\|_{C^{m}(B_1)} \le
\|w_\eps\|_{C^{\gamma+2s}(B_1)} \le C\,\left( [f_\eps]_{C^\gamma(B_2)}
+\int_{B_1^c}\frac{|w_\eps(y)|}{|y|^{n+2s}}\,dy
\right)\\
&\qquad\qquad\qquad\le
C\,\left( \|f_\eps\|_{C^m(B_2)}+\|w_\eps\|_{L^\infty(\R^n)}
\right)\le C\,\|f_\eps\|_{C^m(B_2)}\le C\eps.
\end{split}\end{equation}
Now we define
$$ u_\eps:= v_\eps +\tilde u_{{\bar{R}_\eps}}-w_\eps.$$
Using~\eqref{veps1}, \eqref{weps1} and the consistency result
in Corollary~3.8 of~\cite{POLINOMI}, we see that
$$ (-\Delta)^s v_\eps {\;{\stackrel{0}{=}}\;}0 \quad {\mbox{ and }}\quad
(-\Delta)^s w_\eps{\;{\stackrel{0}{=}}\;}f_\eps \quad {\mbox{ in~$B_1$}}.$$
Thus, 
the consistency result in formula~(1.7)
of~\cite{POLINOMI} implies that
$$
(-\Delta)^s v_\eps {\;{\stackrel{k}{=}}\;}0 \quad {\mbox{ and }}\quad
(-\Delta)^s w_\eps{\;{\stackrel{k}{=}}\;}f_\eps \quad {\mbox{ in~$B_1$}}.$$
Consequently, from~\eqref{u6s8}, we deduce that
$(-\Delta)^s u_\eps\ugu 0+f_\eps-f_\eps$ in~$B_1$,
and this establishes~\eqref{DES:ALL1}.

Furthermore, from~\eqref{90hx:2},
\eqref{90hx} and~\eqref{90hx:1}, we see that
$$ \|u_\e-u\|_{C^m(B_1)}\le
\|v_\e-u\|_{C^m(B_1)}+\|\tilde u_{{\bar{R}_\eps}}\|_{C^m(B_1)}+
\|w_\e\|_{C^m(B_1)}\le \eps+0+C\eps.$$
This proves~\eqref{DES:ALL2} (up to renaming~$\eps$).

Now we take~$R_\eps:=\rho_\eps+{\bar{R}_\eps}$.
{F}rom~\eqref{90hx:3}, \eqref{90hx:9876} and~\eqref{weps1},
we have that, in~$B_{R_\eps}^c$, it holds that~$u_\eps=0+u-0$,
which establishes~\eqref{dtsfgyvhoqwfyguywqegfiowel},
as desired.
\end{proof}

\section{Proof of Theorem~\ref{DIRI}}\label{SF-MAJo2}

First, we prove the existence result
in Theorem~\ref{DIRI}. To this aim, we let~$u_0$ and~$f$
be as in the statement of Theorem~\ref{DIRI} and we
define
$$ u_1:= \chi_{B_2^c}\,u_0\quad
{\mbox{ and }}\quad u_2:=
\chi_{B_2\setminus B_1} \,u_0.$$
We stress that~$u_1$ is smooth in~$B_1$ and~$u_2$
is supported in~$B_2$.

{F}rom Remark~3.5 in~\cite{POLINOMI}, we can write~$(-\Delta)^s
u_1\ugu f_{u_1}$ in~$B_1$, for a suitable function~$f_{u_1}$.

Now we set~$\tilde f:= f-f_{u_1}$ and we consider the
solution of the standard problem
\begin{equation*} \left\{
\begin{matrix}
(-\Delta)^s \tilde u = \tilde f & {\mbox{ in }} B_1,\\
\tilde u=u_2 & {\mbox{ in }} B_1^c.
\end{matrix}
\right.\end{equation*}
Hence, the consistency result in Corollary~3.8
and formula~(1.7) in~\cite{POLINOMI} give that
\begin{equation*} \left\{
\begin{matrix}
(-\Delta)^s \tilde u \ugu \tilde f & {\mbox{ in }} B_1,\\
\tilde u=u_2 & {\mbox{ in }} B_1^c.
\end{matrix}
\right.\end{equation*}
Then, we define~$u:= u_1+\tilde u$
and we see that~$(-\Delta)^s u\ugu f_{u_1}+\tilde f=f$
in~$B_1$. Moreover, in~$B_1^c$ it holds that~$u=u_1+u_2=u_0$,
namely~$u$ is a solution of~\eqref{DIRI2}.
This establishes the existence result 
in Theorem~\ref{DIRI}.
\medskip

Now, we prove the uniqueness claim in Theorem~\ref{DIRI}.
For this, we observe that for any polynomial~$P$ of degree
at most~$k-1$ there exists a unique solution~$u_P$
of the standard problem
\begin{equation}\label{UN} \left\{
\begin{matrix}
(-\Delta)^s u_P = P & {\mbox{ in }} B_1,\\
u_P=0 & {\mbox{ in }} B_1^c.
\end{matrix}
\right.\end{equation}
That is, in view of the
consistency result in Corollary~3.8 of~\cite{POLINOMI},
we have that~$(-\Delta)^s u_P
{\;{\stackrel{0}{=}}\;} P$ in~$B_1$.
Accordingly, from formula~(1.7) in~\cite{POLINOMI},
we get that~$(-\Delta)^s u_P
{\;{\stackrel{k}{=}}\;} P$ in~$B_1$.
Then, by formula~(1.8) in~\cite{POLINOMI},
it follows that~$u_P$ is a solution of
$$ \left\{
\begin{matrix}
(-\Delta)^s u_P \ugu 0 & {\mbox{ in }} B_1,\\
u_P=0 & {\mbox{ in }} B_1^c.
\end{matrix}
\right.$$
This means that if~$u$ is a solution of~\eqref{DIRI2}, then so is~$u+u_P$.

Conversely, if~$u$ and~$v$ are two solutions,
then~$w:=v-u$ satisfies
$$ \left\{
\begin{matrix}
(-\Delta)^s w \ugu 0 & {\mbox{ in }} B_1,\\
w=0 & {\mbox{ in }} B_1^c.
\end{matrix}
\right.$$
This and the consistency result in Lemma~3.9
of~\cite{POLINOMI} (used here with~$j:=0$)
give that~$
(-\Delta)^s w
{\;{\stackrel{0}{=}}\;} P$ in~$B_1$, for some polynomial~$P$
of degree at most~$k-1$.
Hence, using the consistency result in Corollary~3.8
of~\cite{POLINOMI}, we can write
$$ \left\{
\begin{matrix}
(-\Delta)^s w = P & {\mbox{ in }} B_1,\\
w=0 & {\mbox{ in }} B_1^c.
\end{matrix}
\right.$$
{F}rom the uniqueness of the solution
of the standard problem in~\eqref{UN}, we conclude that~$w=u_P$,
and so~$v=u+u_P$.

These observations yield that the space of solutions of~\eqref{DIRI2}
is isomorphic to the space of polynomials~$P$ with degree less than
or equal to~$k-1$, which in turn has dimension~$N_k$,
as given in~\eqref{NK} (see e.g.~\cite{2021arXiv210107941D}).

\section{Proof of Theorem~\ref{NONLI}}\label{SF-MAJo3}

We can extend~$u$ such that~$u\in C^{2m}(B_{1+h})$, for some~$h\in(0,1)$.
Then, for all~$x\in B_{1+h}$, we define~$f(x):=
F\big(x,u(x),\nabla u(x),\dots,D^m u(x)\big)$.
Then, $f\in C^m(B_{1+h})$ and we can exploit Theorem~\ref{DIRI}
and obtain a function~$v\in{\mathcal{U}}_k$ such that
\begin{equation*}
\left\{
\begin{matrix}
(-\Delta)^s v \ugu f \quad{\mbox{ in }}\;B_{1+h},\\
v = 0
\quad{\mbox{ in }}\;B_{1+h}^c.
\end{matrix}
\right.
\end{equation*}
By Theorem~1.3 in~\cite{POLINOMI}, we have that~$
v\in C^m(B_1)$.
Hence, we can set~$w:=u-v\in C^m(B_1)$ and make use of
Theorem~\ref{ALL} to find~$w_\eps$ and~$R_\eps>2$
such that
\begin{eqnarray*}
&& (-\Delta)^s w_\eps \ugu 0 \quad{\mbox{ in }}B_1,\\
&& \| w_\eps-w\|_{C^m({B_1})}\le \eps\\
{\mbox{and }}&& w_\eps=w \quad{\mbox{ in }}B_{R_\eps}^c.
\end{eqnarray*}
Now, we define~$u_\e:= v+w_\e$.
We observe that
$$ (-\Delta)^s u_\eps(x) \ugu  (-\Delta)^s v(x)+ (-\Delta)^s w_\eps(x)
\ugu f(x) =
F\big(x,u(x),\nabla u(x),\dots,D^m u(x)\big)
$$
for all $x\in B_1$.
 
This gives that~\eqref{AMLO-1} is satisfied with
\begin{equation} \label{65654219}\eta_\e(x):= F\big(x,u(x),\nabla u(x),\dots,D^m u(x)\big)-
F\big(x,u_\e(x),\nabla u_\e(x),\dots,D^m u_\e(x)\big).\end{equation}
Moreover, in~$B_{R_\eps}^c$,
$$ u_\e=v+w_\e=v+w=u,$$
and this proves~\eqref{AMLO-4}.

Furthermore,
$$\|u_\e-u\|_{C^m(B_1)}=\|v+w_\e-u\|_{C^m(B_1)}
=\|w_\e-w\|_{C^m(B_1)}\le\e,$$
which establishes~\eqref{AMLO-3}.

Then, we take
$$ S:=2+\sum_{j=0}^m \| D^j u\|_{L^\infty(B_1)}$$
and we denote by~$L$ the Lipschitz norm of~$F$
in~$[-S,S]^{N(m)}$. Thus, employing~\eqref{AMLO-3} and~\eqref{65654219},
for all~$x\in B_1$ we have that
\begin{eqnarray*}&& |\eta_\e(x)|=\Big| F\big(x,u(x),\nabla u(x),\dots,D^m u(x)\big)-
F\big(x,u_\e(x),\nabla u_\e(x),\dots,D^m u_\e(x)\big)\Big|
\\&&\qquad\le L\,\sum_{j=0}^m |D^j u(x)-D^ju_\e(x)|\le
Lm\,\|u_\e-u\|_{C^m(B_1)}\le Lm\e,\end{eqnarray*}
and this gives~\eqref{AMLO-2}, up to renaming~$\e$.

\end{document}